\documentclass[12pt]{amsart}

\author{Paul \textsc{Poncet}}
\address{CMAP, \'{E}cole Polytechnique, Route de Saclay, 91128 Palaiseau Cedex, France \\
and INRIA, Saclay--\^{I}le-de-France}

\email{poncet@cmap.polytechnique.fr}

\usepackage{stmaryrd}
\usepackage{amssymb}
\usepackage{amsmath}
\usepackage{t1enc}
\usepackage{yhmath}
\usepackage{calrsfs} % nouvelles lettres calligraphiques (commande \mathrsfs)
\usepackage{times} % sans que je comprenne pourquoi, les recherches textuelles dans les documents pdf fonctionnent avec cette fonte

\newcommand{\reels}{\mathbb R}

\newcommand{\dint}[1]{\int^{\scriptscriptstyle\infty}_{#1}\!\!}

\newtheorem{theorem}{Theorem}[section]
\newtheorem{corollary}[theorem]{Corollary}
\newtheorem{proposition}[theorem]{Proposition}
\newtheorem{lemma}[theorem]{Lemma}

\theoremstyle{definition}

\newtheorem{example}[theorem]{Example}

\begin{document}

\title{Pseudo-multiplications and their properties}

\date{\today}

 %Nonsmooth analysis


\begin{abstract}
We examine some properties of pseudo-multiplications, which are a special kind of associative binary relations defined on $\overline{\reels}_+ \times \overline{\reels}_+$. 
\end{abstract}

\maketitle

%%%%%%%%%%%%%%%%%%%%%%
%%%%%%%%%%%%%%%%%%%%%%
%%%%%%%%%%%%%%%%%%%%%%
%%%%%%%%%%%%%%%%%%%%%%
\section{Introduction and motivations}

In \cite[Chapter~I]{Poncet11}, we considered the idempotent Radon--Nikodym theorem due to Sugeno and Murofushi \cite{Sugeno87}, and we proved a converse statement in the special case of the Shilkret integral \cite{Shilkret71, Maslov87}. %where the pseudo-multiplication used in the idempotent $\odot$-integral is the usual multiplication $\times$. 
We would like to generalize this converse statement to the \textit{idempotent $\odot$-integral} 
$$
\dint{B} f \odot d\nu, 
$$ 
where $\odot$ is a \textit{pseudo-multiplication} (see e.g.\ Sugeno and Murofushi \cite{Sugeno87}, Benvenuti and Mesiar \cite{Benvenuti04}), i.e.\ a binary relation satisfying a series of natural properties.  
%For this purpose, we need first 
In this note we examine the properties of pseudo-multiplications we shall need to reach this goal. %. This is the goal of this note. 

%These authors showed that, if $\nu$ and $\tau$ be $\sigma$-maxitive measures on a $\sigma$-algebra $\mathrsfs{B}$, with $\tau$ $\sigma$-$\odot$-finite and $\sigma$-principal,  then $\nu$ is $\odot$-absolutely continuous with respect to $\tau$ if and only if there exists some $\mathrsfs{B}$-measurable map $c : E \rightarrow \overline{\reels}_+$ such that 
%$$
%\nu(B) = \dint{B} c \odot d\tau, 
%$$
%for all $B\in \mathrsfs{B}$. 

%%%%%%%%%%%%%%%%%%%%%%%%%%%%%%%%%%%%%%%%%%%%%
%%%%%%%%%%%%%%%%%%%%%%%%%%%%%%%%%%%%%%%%%%%%%
%%%%%%%%%%%%%%%%%%%%%%%%%%%%%%%%%%%%%%%%%%%%%
%%%%%%%%%%%%%%%%%%%%%%%%%%%%%%%%%%%%%%%%%%%%%
\section{Pseudo-multiplications}\label{secmax}

We consider a binary relation $\odot$ defined on $\overline{\reels}_+ \times \overline{\reels}_+$ with the following properties: 
\begin{itemize}
	\item associativity; 
	\item continuity on $(0, \infty) \times [0, \infty]$; %$(0, \infty] \times (0, \infty]$; 
	\item continuity of the map $s \mapsto s \odot t$ on $(0, \infty]$, for all $t$; 
	\item monotonicity in both components; 
	\item existence of a left identity element $1_{\odot}$, i.e.\ $1_{\odot} \odot t =  t$ for all $t$; 
	\item absence of zero divisors, i.e.\ $s \odot t = 0 \Rightarrow 0 \in  \{s, t\}$, for all $s, t$; 
	\item $0$ is an annihilator, i.e.\ $0 \odot t = t \odot 0 = 0$, for all $t$. 
\end{itemize}
We call such a $\odot$ a \textit{pseudo-multiplication}. Note that the axioms above are stronger than in \cite{Sugeno87}, where associativity is not assumed. 

The map $O : \overline{\reels}_+ \rightarrow \overline{\reels}_+$ defined by $O(t) = \inf_{s > 0} s \odot t$ is a \textit{kernel} in the sense that, for all $s, t$: 
\begin{itemize}
	\item $O(s) \leqslant O(t)$ whenever $s \leqslant t$;   
	\item $O(t) \leqslant t$; 
	\item $O(O(t)) = O(t)$. 
\end{itemize}
An element $t$ of $\overline{\reels}_+$ is \textit{$\odot$-finite} if $O(t) = 0$ (and $t$ is \textit{$\odot$-infinite} otherwise). 
We conventionally write $t \ll_{\odot} \infty$ for a $\odot$-finite element $t$. 

\begin{example}
The usual multiplication $\times$ is a pseudo-multi\-pli\-cation with $1_{\times} = 1$ and every element but $\infty$ is $\times$-finite; in this case, the idempotent $\odot$-integral specializes to the Shilkret integral \cite{Shilkret71}. 
The infimum $\wedge$ is a pseudo-multi\-pli\-cation with $1_{\wedge} = \infty$ and every element is $\wedge$-finite; in this case, the idempotent $\odot$-integral specializes to the Sugeno integral \cite{Sugeno74}. 
\end{example}

\begin{lemma}\label{lem:finite}
Given a pseudo-multiplication $\odot$, there exists some positive $\odot$-finite element if and only if $1_{\odot}$ is $\odot$-finite. 
\end{lemma}

\begin{proof}
Assume that $O(t) = 0$ for some $t > 0$. Then $O(1_{\odot}) \odot t \leqslant s \odot 1_{\odot} \odot t = s \odot t$ for all $s > 0$, so that $O(1_{\odot}) \odot t \leqslant O(t) = 0$. Since $\odot$ has no zero divisors, this implies that $O(1_{\odot}) = 0$, i.e.\ $1_{\odot}$ is $\odot$-finite. 
\end{proof}

If $O(1_{\odot}) = 0$, we say that the pseudo-multiplication $\odot$ is \textit{non-degenerate}. This amounts to say that the set of $\odot$-finite elements differs from $\{0\}$. 

\begin{example}\label{ex:deg}
The multiplication $\times$ and the infimum $\wedge$ are non-degenerate pseudo-multiplications. 
The binary relation $\otimes$ defined by $s \otimes t = t$ if $s > 0$ and $0 \otimes t = 0$ is an example of degenerate pseudo-multiplication. 
\end{example}

\begin{proposition}\label{lem:inv}
Given a non-degenerate pseudo-multiplication $\odot$, the following conditions are equivalent for an element $t \in \overline{\reels}_+$: 
\begin{itemize}
	\item $t$ is $\odot$-finite;
	\item $s \odot t \ll_{\odot} \infty$ for some $s > 0$; 	
	\item $s \odot t \leqslant 1_{\odot}$ for some $s > 0$; 
	\item $t \odot s' \leqslant 1_{\odot}$ for some $s' > 0$; 
	\item $t \odot s' \ll_{\odot} \infty$ for some $s' > 0$.  	
\end{itemize}
\end{proposition}

\begin{proof}
Assume that $t$ is $\odot$-finite. Then $s \odot t \ll_{\odot} \infty$ with $s = 1_{\odot}$. 

Assume that $s \odot t \ll_{\odot} \infty$ for some $s > 0$. Then $\lim_{s' \rightarrow 0^+} s' \odot s \odot t = 0$, so there exists $s' > 0$ such that $(s' \odot s) \odot t \leqslant 1_{\odot}$. 

Assume that $s \odot t \leqslant 1_{\odot}$ for some $s > 0$. Then $s \odot t \odot s' \leqslant 1_{\odot} \odot s' = s'$ for all $s' > 0$, so that $\lim_{s' \rightarrow 0^+} s \odot t \odot s' = 0$. By continuity and the fact that $\odot$ has no zero divisors we get $\lim_{s' \rightarrow 0^+} t \odot s' = 0$. This implies that $t \odot s' \leqslant 1_{\odot}$ for some $s' > 0$. 

Assume that $t \odot s' \ll_{\odot} \infty$ for some $s' > 0$. Then $\lim_{s \rightarrow 0^{+}} s \odot t \odot s' = 0$, so that $O(t) \odot s' = 0$. This shows that $O(t) = 0$, i.e.\ $t$ is $\odot$-finite. 
\end{proof}

%\begin{lemma}\label{lem:id}
%In the non-degenerate case, the following conditions are equivalent for an element $t > 0$: 
%\begin{enumerate}
%	\item\label{id1} $t$ is idempotent and $\odot$-infinite;  
%	\item\label{id2} $O(t) = t$; 
%	\item\label{id4} $s \odot t = t$ for all $0 < s < t$. 
%\end{enumerate}
%\end{lemma}

%\begin{proof}
%We clearly have (\ref{id4}) $\Rightarrow$ (\ref{id1}) and (\ref{id4}) $\Rightarrow$ (\ref{id2}). Let us show that (\ref{id1}) $\Rightarrow$ (\ref{id2}). So let $t$ be idempotent and $\odot$-infinite. If $O(t) < t$, there exists some $s > 0$ such that $s \odot t < t$. 
%\end{proof}

\begin{lemma}\label{lem:odotfi}
Given a pseudo-multiplication $\odot$, the set of $\odot$-finite elements is either $\{0\}$, or $[0, \infty]$, or of the form $[0, \phi)$ for some $\phi \in (1_{\odot}, \infty]$ such that $O(\phi) = \phi$. 
\end{lemma}

\begin{proof}
Denote by $F_{\odot}$ the set of $\odot$-finite elements, and assume that $F_{\odot} \neq \{0\}$. 
Then, by Lemma~\ref{lem:finite}, $F_{\odot}$ is an interval containing $0$ and $1_{\odot}$, so it satisfies $[0, \phi) \subset F_{\odot} \subset [0, \phi]$ for some $\phi \geqslant 1_{\odot}$. 
If $\phi = \infty$ we have $F_{\odot} = [0, \infty)$ or $F_{\odot} = [0, \infty]$. In the case where $F_{\odot} = [0, \infty)$ the lemma asserts also that $O(\phi) = \phi$. We have $O(\phi) \leqslant \infty = \phi$. Suppose that $O(\phi) < \phi$. Then $O(\phi) \in [0, \phi) \subset F_{\odot}$, so $O(\phi)$ is $\odot$-finite. This means that $O(O(\phi)) = 0$, hence $O(\phi) = 0$. So $\phi$ is $\odot$-finite, a contradiction. 

Now suppose that $\phi < \infty$. 
%It remains to show that $\phi \notin F_{\odot}$. 
%So assume that $\phi \in F_{\odot}$. 	
Since $\phi$ is the supremum of $F_{\odot}$, there exists a decreasing sequence $(t_n)$ of $\odot$-infinite elements that tends to $\phi$. 
Let $t_n' = O(t_n)$. Since $O(t_n') = O(O(t_n)) = O(t_n) = t_n' > 0$, $t_n'$ is $\odot$-infinite, hence $\phi \leqslant t_n' \leqslant t_n$ for all $n$. This implies that $t_n' \rightarrow \phi$ when $n \rightarrow \infty$. 
%The properties of the map $O$ imply that we can assume $t_n = O(t_n)$ for all $n$, without loss of generality. 
Again, $O(t_n') = t_n'$, so that $s \odot t_n' = t_n'$ for all $0 < s < 1_{\odot}$ and all $n$. With $n \rightarrow \infty$, we obtain $s \odot \phi  = \phi$ for all $0 < s < 1_{\odot}$. This shows that $O(\phi) = \phi > 0$, so $\phi$ is $\odot$-infinite, and $F_{\odot} = [0, \phi)$.  
\end{proof}

\begin{example}
We have $F_{\times} = [0, \infty)$, $F_{\wedge} = [0, \infty]$, and $F_{\otimes} = \{0\}$ (see the definition of $\otimes$ in Example~\ref{ex:deg}). 
Given $0 < \phi < \infty$ we can also build an example of pseudo-multiplication $\odot_{\phi}$ with $F_{\odot_{\phi}} = [0, \phi)$ as follows. 
If $\tanh$ denotes the hyperbolic tangent, we define $s \odot_{\phi} t = \max(s, t)$ if $s \geqslant \phi$ or $t \geqslant \phi$, and $s \odot_{\phi} t = \phi \tanh(\tanh^{-1}(s/\phi) \tanh^{-1}(t/\phi))$ otherwise. 
\end{example}

%\begin{remark}\label{rk:phi}
%The proof shows that $s \odot \phi = \phi$ for all $0 < s < 1_{\odot}$. A similar argument shows that $\phi \odot s = \phi$ for all $0 < s < 1_{\odot}$. When $s \rightarrow 1_{\odot}$, we get $\phi \odot 1_{\odot} = \phi$. \textcolor{red}{BOFBOFBOF...}
%\end{remark}

\begin{lemma}\label{lem:com}
A pseudo-multiplication $\odot$ is non-degenerate if and only if the monoid $([0, 1_{\odot}], \odot)$ is commutative. 
\end{lemma}

\begin{proof}
Assume that $([0, 1_{\odot}], \odot)$ is commutative. Then 
\begin{align*}
O(1_{\odot}) &= \inf_{s \in (0, 1_{\odot}]} s \odot 1_{\odot} = \inf_{s \in (0, 1_{\odot}]} 1_{\odot} \odot s \\
             &= \inf_{s \in (0, 1_{\odot}]} s = 0.
\end{align*}
This shows that $\odot$ is non-degenerate. 

Conversely, assume that $\odot$ is non-degenerate. We show that, in $[0, 1_{\odot}]$, the binary relation $\odot$ is continuous; this will imply, by Faucett's theorem \cite[Lemma~5]{Faucett55b}, that $([0, 1_{\odot}], \odot)$ is commutative. 
So let $s_n \rightarrow s$ and $t_n \rightarrow t$ in $[0, 1_{\odot}]$. We prove that $s_n \odot t_n \rightarrow s \odot t$. If $s > 0$ this is a consequence of the properties of the pseudo-multiplication $\odot$, so suppose that $s = 0$, and let us show that $s_n \odot t_n \rightarrow 0$. 
For this purpose, let $\varepsilon > 0$, and let $s_n' = s_n \oplus \varepsilon$. Since $s_n'$ tends to $\varepsilon > 0$, the sequence $s_n' \odot t_n$ tends to $\varepsilon \odot t$. But $s_n' \odot t_n = (s_n \oplus \varepsilon) \odot t_n = (s_n \odot t_n) \oplus (\varepsilon \odot t_n)$, and taking the limit superior gives $\varepsilon \odot t = \ell \oplus (\varepsilon \odot t)$, where $\ell := \limsup (s_n \odot t_n)$. This shows that $\ell \leqslant \varepsilon \odot t$, for all $\varepsilon > 0$. In other words, $\ell \leqslant O(t)$. Moreover, we are in the non-degenerate case, and $t \leqslant 1_{\odot}$, so that $O(t) = 0$. Hence, $\ell = 0$, so that $\lim (s_n \odot t_n)  = 0$, which is the desired result. 
\end{proof}

We summarize the previous lemmata in a single theorem:  

\begin{theorem}\label{lem:phi}
Given a pseudo-multiplication $\odot$, the following conditions are equivalent: 
\begin{itemize}
	\item $\odot$ is non-degenerate, i.e.\ $1_{\odot}$ is $\odot$-finite; 
	\item there exists some positive $\odot$-finite element; 
	\item the monoid $([0, 1_{\odot}], \odot)$ is commutative; 
	\item the set $F_{\odot}$ of $\odot$-finite elements is either $[0, \infty]$ or of the form $[0, \phi)$ for some $\phi \in (1_{\odot}, \infty]$. 
\end{itemize}
Moreover, in the case where $F_{\odot} = [0, \phi)$, % for some $\phi \in (1_{\odot}, \infty]$. 
then $\phi$ satisfies $O(\phi) = \phi$ and $t \odot \phi = \phi \odot t = \phi$, for all $0 < t \leqslant \phi$. In particular, $\phi$ is idempotent. %	, i.e.\ $\phi \odot \phi = \phi$. 
\end{theorem}

\begin{proof}
Let $0 < t < \phi$. Since $t$ is $\odot$-finite, there is some $s' > 0$ such that $t \odot s' \leqslant 1_{\odot}$ by Proposition~\ref{lem:inv}. 
Thus, $\phi = O(\phi) \leqslant t \odot \phi = t \odot O(\phi) \leqslant t \odot (s' \odot \phi) = (t \odot s') \odot \phi \leqslant 1_{\odot} \odot \phi = \phi$. We obtain that $t \odot \phi = \phi$, and when  $t \rightarrow \phi$ we get $\phi \odot \phi = \phi$, i.e.\ $\phi$ is idempotent. 
Using again Proposition~\ref{lem:inv}, the fact that $\phi$ be $\odot$-infinite implies $\phi \odot s \geqslant \phi$ for all $s > 0$. Hence, if $0 < t < \phi$, we get $\phi = \phi \odot \phi \geqslant \phi \odot t \geqslant \phi$, i.e.\ $\phi \odot t = \phi$. 
%Let $0 < t < \phi$. If $t < 1_{\odot}$, the result is clear by Remark~\ref{rk:phi}, so assume that $t \geqslant 1_{\odot}$. Since $t$ is $\odot$-finite, there is some $s > 0$ such that $s \odot t \leqslant 1_{\odot}$ by Lemma~\ref{lem:inv}. 
%Note that we can choose $s \leqslant 1_{\odot}$. 
%Thus, by Lemma~\ref{lem:odotfi} and Remark~\ref{rk:phi}, $\phi = \phi \odot 1_{\odot} \leqslant \phi \odot t = (\phi \odot s) \odot t = \phi \odot (s \odot t) \leqslant \phi \odot 1_{\odot} = \phi$. We obtain that $\phi \odot t = \phi$, and when $t \rightarrow \phi$ we get $\phi \odot \phi = \phi$, i.e.\ $\phi$ is idempotent. Moreover, $\phi = 1_{\odot} \odot \phi \leqslant t \odot \phi \leqslant \phi \odot \phi = \phi$. This shows that $t \odot \phi = \phi$ for all $0 < t < \phi$. 
\end{proof}

\begin{corollary}\label{lem:dec}
Given a pseudo-multiplication $\odot$, it is not possible to find $t < \phi$ and $t' > \phi$ such that $t \odot t' = \phi$, if $\phi$ denotes the supremum of the set of $\odot$-finite elements. 
\end{corollary}

\begin{proof}
Again, let $F_{\odot}$ be the set of $\odot$-finite elements. The result is clear if $F_{\odot} = \{ 0 \}$ or $F_{\odot} = [0, \infty]$. Now if $F_{\odot} = [0, \phi)$ for some $\phi \in (1_{\odot}, \infty]$, we write $\phi = t \odot t'$ with $t < \phi$. Then $\phi = \phi \odot \phi = \phi \odot (t \odot t') = (\phi \odot t) \odot t' = \phi \odot t'$ by Theorem~\ref{lem:phi}. Since $\phi > 1_{\odot}$, this implies $\phi \geqslant 1_{\odot} \odot t' = t'$. 
\end{proof}

\bibliographystyle{plain}%{alpha}
%\bibliography{C:/LocalTex/BIBLIO}

\def\cprime{$'$} \def\cprime{$'$} \def\cprime{$'$} \def\cprime{$'$}
  \def\ocirc#1{\ifmmode\setbox0=\hbox{$#1$}\dimen0=\ht0 \advance\dimen0
  by1pt\rlap{\hbox to\wd0{\hss\raise\dimen0
  \hbox{\hskip.2em$\scriptscriptstyle\circ$}\hss}}#1\else {\accent"17 #1}\fi}
  \def\ocirc#1{\ifmmode\setbox0=\hbox{$#1$}\dimen0=\ht0 \advance\dimen0
  by1pt\rlap{\hbox to\wd0{\hss\raise\dimen0
  \hbox{\hskip.2em$\scriptscriptstyle\circ$}\hss}}#1\else {\accent"17 #1}\fi}

\end{document}